\documentclass[letterpaper, 10pt, conference]{ieeeconf}  

\IEEEoverridecommandlockouts                              
\usepackage{bbold}

\usepackage[sort, nocompress]{cite}
\usepackage{amsmath,amssymb,amsthm, amsfonts}
\usepackage{graphicx}
\usepackage{xcolor}

\usepackage{cleveref}
\usepackage{footmisc}
\usepackage{csquotes}

\def\BibTeX{{\rm B\kern-.05em{\sc i\kern-.025em b}\kern-.08em
    T\kern-.1667em\lower.7ex\hbox{E}\kern-.125emX}}

\newcommand{\ie}{\textit{i.e. }}
\newcommand{\eg}{\textit{e.g. }}
    
\newcommand{\R}{\mathbb R}

\newcommand{\exit}{\Gamma}
\newcommand{\minf}{\Phi}
\newcommand{\entrance}{\Psi}
\newcommand{\maxf}{\Theta}

\newcommand{\Hess}[1]{\operatorname{Hess}{#1}}
\newcommand{\grad}[1]{\nabla{#1}}
\newcommand{\tr}{\operatorname{tr}}

\newcommand{\FwdInv}{F} 
\newcommand{\FwdInvExit}{G}
\newcommand{\FwdConv}{Q}
\newcommand{\FwdConvExit}{N}

\newcommand{\lemmaF}{U} 
\newcommand{\lemmaT}{H} 

\newcommand{\ave}{\rho}
\newcommand{\var}{\zeta}

\theoremstyle{plain}
\newtheorem{theorem}{Theorem}

\newtheorem{lemma}{Lemma}

\newtheorem{definition}{Definition}

\newtheorem{example}{Example}

\newcommand{\mE}{\ensuremath{\mathbf{E}}}

\newcommand{\mP}{\ensuremath{\mathbf{P}}}

\newcommand\gC{{\mathcal{C}}}

\newcommand\gE{{\mathcal{E}}}

\newcommand\gK{{\mathcal{K}}}
\newcommand\gL{{\mathcal{L}}}
\newcommand\gM{{\mathcal{M}}}

\newcommand\gX{{\mathcal{X}}}
\newcommand\gY{{\mathcal{Y}}}

\newcommand{\norm}[1]{\left\lVert#1\right\rVert}

\providecommand{\keywords}[1]{\vspace*{0.5\baselineskip}\par\noindent\textbf{Keywords:} #1}

{%
\begin{list}{#1}{
\vspace{-\topsep}
\vspace{-\partopsep}
\setlength{\itemindent}{0cm}
\setlength{\rightmargin}{0cm}
\setlength{\listparindent}{0cm}
\settowidth{\labelwidth}{#1}
\setlength{\leftmargin}{\labelwidth}
\addtolength{\leftmargin}{\labelsep}
\setlength{\itemsep}{0cm}
}%
}%
{%
\end{list}
\vspace{-\topsep}
\vspace{-\partopsep}
}

{\begin{enumerate}%
\renewcommand{\theenumi}{\roman{enumi}}%
}%
{\end{enumerate}}

\theoremstyle{plain}

\newtheorem*{lemma*}{Lemma}

\newtheorem*{prop*}{Proposition}

\theoremstyle{definition}

\newtheorem*{defn*}{Definition}

\newtheorem*{exmp*}{Example}

\newtheorem*{conj*}{Conjecture}

\theoremstyle{remark}

\newtheorem*{rmk*}{Remark}

\title{\LARGE \bf Safe Control in the Presence of Stochastic Uncertainties}
\author{Albert Chern\(^{1}\), Xiang Wang\(^{2}\), Abhiram Iyer\(^{2}\), Yorie Nakahira\(^{2,*}\)
\thanks{\(^1\)A. Chern is with the Department of Computer Science and Engineering, University of California San Diego, 9500 Gilman Dr, La Jolla, CA 92093, USA, {\tt\small alchern@ucsd.edu}.} 
\thanks{\(^{2}\)X. Wang, A. Iyer and Y. Nakahira are with the Department of Electrical and Computer Engineering Carnegie Mellon Universty, 5000 Forbes Ave, Pittsburgh, PA 15213, USA,
        {\tt\small \{xiangw2,abhirami,ynakahir\}@andrew.cmu.edu}.}
         \thanks{\(^{*}\)To whom correspondence should be addressed.}
}

\begin{document}
\maketitle
\thispagestyle{empty}
\pagestyle{empty}

\begin{abstract}

Accurate quantification of safety is essential for the design of autonomous systems. In this paper, we present a methodology to characterize the exact probabilities associated with invariance and recovery in safe control. We consider a stochastic control system where control barrier functions, gradient-based methods, and barrier certificates are used to constrain control actions and validate safety. We derive the probability distributions of the minimum and maximum barrier function values during any time interval and the first entry and exit times to and from any super level sets of the barrier function. These distributions are characterized by deterministic convection-diffusion equations, and the approach used is generalizable to other safe control methods based on barrier functions. These distributions can be used to characterize various quantities associated with invariance and recovery, such as the safety margin, the probability of entering and recovering from safe and unsafe regions, and the mean and tail distributions of failure and recovery times. 

\end{abstract}

\section{Introduction}

Safe control and verification methods based on barrier functions are widely used in autonomous systems. Examples of these methods include control barrier functions (CBFs)~\cite{ames2019control, clark2019control}, gradient-based modifications~\cite{khatib1986real, gracia2013reactive}, and barrier certificates~\cite{prajna2007framework}. Although significant prior work has been dedicated to exploring these concepts in noiseless or deterministic systems, much less literature has focused on the stochastic setting, largely due to the complexities in dealing with the dynamic evolution of distributions associated with coupled system states, control actions, and barrier function values. For instance, in a deterministic framework, we can design the system state to move toward the interior of the safe set at the boundary or outside of the safe region. Conversely, in a stochastic framework, we need to additionally characterize the dynamics of the (tail) distributions as well as evaluating the minimum and maximum values of the barrier function values over a time interval, which is more demanding than evaluating the function values at a fixed time.

Given the challenges in obtaining exact probability, existing literature has provided insightful upper and lower bounds for safe probabilities \cite{clark2019control,prajna2007framework,yaghoubi2020risk, santoyo2021barrier,cheng2020safe}. Their approaches can be roughly classified into ones based on martingale inequalities and ones that convert a stochastic problem into a deterministic one.\footnote{The two approaches can also be used in concert.} However, these bounds may be loose and thus lead to unnecessarily conservative control actions that compromise performance. 

To bound the probability of safety, the former approach constructs super/sub-martingales that bound the evolution of the barrier function values over time~\cite{clark2019control,prajna2007framework,yaghoubi2020risk, santoyo2021barrier}. In this approach, the constructed martingales do not use the complete distribution of the system dynamics, thus resulting in a loose bound. The latter approach finds the tail probability of system variables being larger or smaller than a threshold, and then considers the worst-case performance when the system variables are bounded by the threshold~\cite{clark2019control,cheng2020safe}. As this approach may not capture the variability in the probability density of the system variables, the resulting bound may also be conservative.  

Such existing bounds for safety probability may be too conservative to guide the risk-control and decision-making processes in real-time. In extreme situations with a highly limited set of admissible control actions, a conservative constraint may be too strong to have a feasible solution. Therefore, accurate quantification for safe probabilities is of critical importance in the design autonomous systems. 

\subsection{Objectives and contributions of this paper}

Motivated by the need for rigorous safe probability quantification, we study the stochastic dynamics and behaviors of the random processes associated with system safety in both space and time. We consider the settings of control barrier functions (CBFs)~\cite{ames2019control, clark2019control}, gradient-based modifications~\cite{khatib1986real, gracia2013reactive}, and barrier certificates~\cite{prajna2007framework}. 

In particular, we list the contributions of this paper below. 
\begin{enumerate}
\item We characterize the exact distributions for the minimum and maximum values of the barrier functions during any given time interval and the first entry and exit times to and from the safe region and any super level sets of the barrier functions.
\item These distributions are given as the solutions to the deterministic convection-diffusion equations. These equations can be solved using standard tools from numerical algorithms. 
\item The analysis framework builds upon the proof techniques from the Feynman-Kac representation (\cite[Theorem 1.3.17]{pham2009continuous} and Lemma~\ref{thm:Invariant Probability}) to derive a few of its variations (Lemma~\ref{lem:Lemma1}, Lemma~\ref{lem:EscapeTime}, and Lemma~\ref{lem:EscapeTimeLemma3}). These variations can be used to derive the distributions of other random processes beyond the settings studied in this paper.
\item When the state initiates from the safe region, these distributions can be used to infer the probability of failure in any finite time horizon, the distribution of the safety margin in normal operation, and the probability of the time to exit from the safe region. 
\item When the state is outside of the safe region, these distributions can be used to infer the probability of recovery in any finite time horizon, the distribution of the distance from the safe region in the recovery process, and the distribution of the recovery time to re-enter the safe region. 
\end{enumerate}

The information on these probabilities can be used to better account for uncertainties in the system model (\eg unmodeled dynamics), control process (\eg sensing and actuation noise), and environment (\eg moving objects). A few examples of these scenarios include when unmodeled dynamics are captured using statistical models such as Gaussian processes~\cite{fan2019,fisac2018general}, when sensing and actuation noise is modeled by multivariate normal distributions~\cite{zhou1996robust}, or when perception and data fusion results come with some error probabilities~\cite{ferguson2008, leon2019}. 

Beyond the settings focused in this paper, an exact characterization of these tail distributions can be used to formulate new optimization problems with probabilistic bounds on safety or recovery, which can be solved using existing techniques from partial differential equations (PDE) constrained optimization. These techniques will guide real-time decision-making and risk-control in many autonomous systems~\cite{koopman2017autonomous,tadele2014safety,moustris2011evolution}.  


\subsection{Notation}
Let $\R$, $\R_+$, $\R^n$, and $\R^{m\times n}$ be the set of real numbers, the set of non-negative real numbers, the set of $n$-dimensional real vectors, and the set of $m \times n$ real matrices. Let $x[k]$ be the $k$-th element of vector $x$. 
Let $x^+ = \max(x,0)$ and $x \wedge y = \min\{x,y\}$ for \(x, y \in \R\). Let $f:\gX \rightarrow \gY$ represent that $f$ is a mapping from space $\gX$ to space $\gY$. 
Let \(\mathbb{1}_{\gE}\) be an indicator function, which takes \(1\) when condition \(\gE\) holds and \(0\) otherwise. 
Let \(\mP_x(\gE) = \mP(\gE(X)|X_0=x)\) denote the probability of event \(\gE(X)\) involving a stochastic process \( X = \{X_t\}_{t\in\R_+}\) conditioned on
\(X_0=x\). Let \(\mE_x [ F ( X ) ] = \mE[ F ( X ) | X_0 = x ]\) denote the expectation of $F ( X )$ (a functional of $X$) conditioned on $X_0 = x$. We use upper-case letters (\eg $X$) to denote random variables and lower-case letters (e.g. $x$) to denote their specific realizations.

\section{Problem Statement}
We are interested in identifying the exact probability that a system state enters or leaves the safe region in a given time interval as well as the distributions of quantities associated with regular operation (\eg safety margin and failure time) and recovery process (distance from the safe set and speed of recovery). To do so, we describe the system dynamics and the related safety properties in subsection~\ref{sec:SystemDescription}. Then, we will introduce the safe control methods to be analyzed in subsection~\ref{sec:SafeControl}.

\subsection{System Description and design specifications} 
\label{sec:SystemDescription}
We consider a control system with stochastic noise of $k$-dimensional Brownian motion $W_t$ starting from $W_0 = 0$. The state of the control system, $X_t \in \R^{n}, t \in \R_+$, is the strong solution of the following stochastic differential equation (SDE)
\begin{equation}\label{eq:x_trajectory}
    dX_t = (f(X_t) + g(X_t)U_t)dt + \sigma(X_t)dW_t,
\end{equation}
where $U_t \in \R^{m}$ is the control input.  Throughout this paper, we assume sufficient regularity in the coefficients of the SDE~\eqref{eq:x_trajectory}. That is, $f, g, (U_t)_{t \in \R_+}$ are chosen in a way such that there exist a unique strong solution to \eqref{eq:x_trajectory}.\footnote{ See~\cite[Chapter~5]{moustris2011evolution},~\cite[Chapter~1]{oksendal_stochastic_2003a}, \cite[Chapter II.7]{borodin_stochastic_2017} and references therein for required conditions.\label{ft:FN-Solution}}
The state of the system $X_t$ can contain both the state of the plant, which can be controlled, and the state of environmental variables (\eg moving obstacles), which cannot be controlled.\footnote{Let \(X_p\) be the state of the plant to be controlled and \(X_o\) be the environmental variables which cannot be controlled. The dynamics of \(X_p\) and \(X_o\) can be jointly captured by a system with the augmented state space $X^\intercal = [X^\intercal_p, X_o^\intercal]$.\label{fn:augument_x}} The size of $\sigma(X_t)$ is determined from the uncertainties in the disturbance, unmodeled dynamics~\cite{cheng2020safe,dhiman2020control,fan2019,fisac2018general}, and the prediction errors of the environmental variables~\cite{lefevre2014survey,ferguson2008}. Examples of these cases include when the unmodeled dynamics are captured using statistical models such as Gaussian Processes\footnote{In such settings, the value of $\sigma(X_t)$ can be determined from the output of the Gaussian Process as in~\cite{cheng2020safe,dhiman2020control,fan2019,fisac2018general}} and when the motion of the environment variables are estimated using physics-based models such as Kalman filters~\cite{lefevre2014survey,ferguson2008}.

The control signal is generated by a memoryless state feedback controller of the form 
\begin{align} 
U_t = K ( X_t ).
\end{align} 
Examples of controllers of this form include classic controllers (\eg PID~\cite{franklin2002feedback}), optimization-based controllers (\eg some robust controllers~\cite{zhou1996robust} and nonlinear controllers~\cite{khali1996adaptive}), safe control methods~\cite{clark2019control,prajna2007framework,yaghoubi2020risk, santoyo2021barrier,cheng2020safe,agrawal2017discrete}, and some hybrid of those. 

One of the major design considerations for $K$ is system safety. This in turn requires a rigorous analysis on the following two aspects related to safety: forward invariance and recovery. The forward invariance refers to the property that the system state always stays in the safe region, while recovery refers to the property that the system state recovers to the safe region even if it originates from the unsafe region.
\begin{definition}[Safe Set~\cite{ames2019control}]\label{def:safe_region}
The safe region is defined using a set $\gC$ that is characterized by a super level set of some function $\phi(x)$, \ie
\begin{equation}\label{eq:safe_region}
\mathcal{C} =\left\{x \in \mathbb{R}^{n}: \phi(x) \geq 0\right\},
\end{equation}
We call $\gC$ the \textit{safe set}, and \(\phi(x)\) the barrier function.
Additionally, we denote the boundary of the safe set and unsafe set by
\begin{align}
    \partial \gC &= \{x\in\R^n:\phi(x)=0\}, \\    \gC^c & =\{x\in\R^{n}: \phi(x)<0\}.
\end{align}
We assume that $\phi(x): \mathbb{R}^n \rightarrow \mathbb{R}$ is a second-order differentiable function whose gradient does not vanish at $\partial \gC$.
\end{definition}
We assume that the dynamics and uncertainty in the control system and the environment (\eg uncertainties in the environment and moving obstacles) are all captured in~\eqref{eq:x_trajectory}\footref{fn:augument_x} so that the safe set can be defined using a static function $\phi(x)$ instead of a time-varying one.
\subsubsection{Forward invariance} When the initial state starts from a safe region, \ie $X_0 = x\in\gC$, we study the quantities associated with forward invariance: the safety margin from $\partial \gC$ and the first exit time from $\gC$.

\begin{definition}[Forward Invariant Set~\cite{khali1996adaptive}]
A set $\gM$ is said to be a forward invariant set\footnote{A forward invariant set is also known as positive invariant set in the existing literature~\cite{khali1996adaptive}.} with respect to system \eqref{eq:x_trajectory} if 
\begin{align}
    X_0 \in \gM \Rightarrow X_t \in \gM,\quad  \forall t \geq 0.
\end{align}
\end{definition}
In other words, if the state belongs to a forward invariant set $\gM$ with respect to \eqref{eq:x_trajectory} at some time, then it remains in $\gM$ for all future times. 

We study the forward invariance of the safe set through characterizing the distributions of the following two random variables:
\begin{align}
\label{eq:min_phi}
   \minf_x(T) &:= \inf\{ \phi(X_{t}) \in \R : t \in [0,T] , X_0 = x  \},\\
\label{eq:first_exit_time}
   \exit_x(\ell)  &:= \inf\{t \in \R_+ :  \phi(X_t)\leq \ell , X_0 = x \},
\end{align}

\noindent where $x \in \gC$. The value of $\minf_x(T)$ informs the worst-case safety margin from the boundary $\partial \gC$ during $[0, T]$, while the value of $\exit_x(0)$ is the first exit time from the safe set.


\subsubsection{Recovery} When the initial state is outside of the safe region, \ie$X_0 = x \notin \gC$, we study the quantities associated with recovery: the distance from $\gC$ and the recovery time. 

Properties associated with entrance to a set have been studied from the perspective of stochastic stability in different contexts.\footnote{Stochastic stability is defined by, given $\gM$, the existence of a set $\gM_0$ such that
$x_0 \in \gM_0 \Rightarrow x_t \in \gM , \forall t \geq 0$ with probability one~\cite[Chapter~2]{kushner1967stochastic}. Stochastic stability requires invariance with probability one, and the invariance only need to hold for some set of (not all) initial conditions.} However, stochastic stability is not applicable in our setting for two reasons: the state rarely stays in a set indefinitely after reaching it; and the initial state is often predetermined and cannot be chosen.\footnote{From Theorem \ref{thm:InvariantProbability_MainTheorem1}, it can be shown that $\mP ( \maxf_x(T) \geq 0 ) \rightarrow 0$ as $T\rightarrow \infty$ in many cases.}
Instead, we consider a weaker notion than stability, formally defined below, which requires re-entry for any initial condition but does not require invariance upon re-entry. 

\begin{definition}[Forward Convergent Set]
A set $\gM$ is said to be a forward convergent set with respect to system~\eqref{eq:x_trajectory} if 
\begin{align}\label{eq:ProblemStatement_ConvergenceSet}
    X_{0} \notin \gM  \Rightarrow \exists \tau  \in (0, \infty) \text{ s.t. } X_{\tau} \in \gM .
\end{align}
\end{definition}
That is, for any initial state that does not belong to the safe set, the state re-enters the safe set in finite time. Equivalently, even if the state does not belong to a forward convergent set $\gM$ with respect to \eqref{eq:x_trajectory} at some time, it will enter $\gM$ in finite time.\footnote{This equivalence holds when the system parameters $f, g, \sigma$, and control policy $K$ are time-invariant mappings.} Note that a forward convergent set does not require the state to stay in the set indefinitely after re-entry. 

We study forward convergence to the safe set through characterizing the distribution of the following two random variables:
\begin{align}
\label{eq:max_phi}
   \maxf_x(T) & := \sup\{ \phi(X_{t}) \in \R : t \in [0,T] , X_0 = x  \},\\
\label{eq:first_entry_time}
   \entrance_x(\ell) & := \inf\{t \in \R_+ :  \phi(X_t)\geq \ell , X_0 = x \},
\end{align}
\noindent where $x \notin \gC$. The value of $\maxf_x(T)$ informs the distance\footnote{
Here, we refer to distance with a slightly light abuse of notation as the value of \(\maxf_x(\ell),\minf_x(\ell)\) measure the value of function $\phi$, instead of the distance between a state \(x\) to the safe set \(\gC\). Note that the following condition holds: 
\[
\phi(x)\geq 0 \Leftrightarrow Dist(x,\gC):= \inf_{y\in\gC} \norm{x-y}=0
\]\label{ft:distance_meansure}}
to the safe set, while the recovery time $\entrance_x(0)$ is the duration required for the state to re-enter the safe region for the first time.

The probability of staying within the safe set $\gC$ during a time interval $[0 , T]$ given $X_0 = x \in \gC$ can be computed using the distributions of \eqref{eq:min_phi} or \eqref{eq:first_exit_time} as
\begin{align} 
\label{eq:safe_prob_invariant}
\begin{split}
    \mP_x(X_t\in\gC, \forall t\in[0,T]) &= \mP \left(\minf_x(T) \geq 0 \right) \\
    &= 1 - \mP \left(\exit_x(0) \leq  T \right).
\end{split}
\end{align}
Similarly, the probability of having re-entered the safe set $\gC$ by time $T$ given $X_0 = x \notin \gC$ can be computed using the distributions of \eqref{eq:max_phi} or \eqref{eq:first_entry_time} as  
\begin{align} 
\label{eq:safe_prob_convergent}
\begin{split}
 \mP_x\left( \exists t \in [0, T] \text{ s.t. } X_t \in\gC \right) &= \mP \left( \maxf_x(T) \geq 0 \right) \\
    &= 1 - \mP \left( \entrance_x(0) \leq  T \right).
\end{split}
\end{align}
More generally, the distributions of $\minf_x(T)$, $\maxf_x(T)$, \(\Gamma_x(\ell)\), and $\entrance_x(\ell)$ contain much richer information than the probability of forward invariance or recovery. Thus, knowing their exact distributions allows uncertainty and risk to be rigorously quantified in autonomous control systems.


\subsection{Safe control methods to be analyzed}
\label{sec:SafeControl}

The safety of the system can be evaluated using \(\phi(X_t)\), whose dynamics are captured in the expectation and higher-order moments of \(d\phi(X_t)\). Most existing literature decides control actions and/or verify safety using this expectation, which are characterized using the infinitesimal generator.


\begin{definition}[Infinitesimal Generator]
The infinitesimal generator $A$ of
a stochastic process $\{ Y_t  \in \R^n \}_{t \in \R_+}$ is
\begin{equation}\label{eq:InfinitesimalGenerator}
\begin{split}
AF\left(y\right) & = \lim _{h\rightarrow 0} \frac{\mE_y\left[F(Y_{h})\right]-F\left(y\right)}{h}\\
\end{split}
\end{equation}
whose domain is the set of all functions $F: \R^n \rightarrow \R$ such that the limit of \eqref{eq:InfinitesimalGenerator} exists for all $y \in \R^n$. 
\end{definition}

Applying the infinitesimal generator to $\phi$ yields\footnote{This formula is computed using the It\^o's Lemma, stated below: Given a $n$-dimensional real valued diffusion process and any twice differentiable scalar function $f: \R^n \rightarrow \R$, It\^o's lemma states that
$
df= \left(\gL_\mu f + \frac{1}{2}\tr\left(\sigma\sigma^\intercal \Hess{f}\right)\right)\, dt + \grad{f}\cdot\sigma\,dW_{t},
$
where $\gL_\mu{f}(x) = \mu(x)\cdot\grad{f}(x)$ is the Lie derivative of \(f(x)\) along the vector field \(\mu(x)\)~\cite{karatzas2014brownian}.} 
\begin{align}
\label{eq:inf_gen_phi1}
       A\phi & = \gL_f\phi + \gL_{g}\phi u + \frac{1}{2}\tr\left(\sigma\sigma^\intercal\Hess{\phi}\right)\\
        & =: D_\phi(x,u)
\end{align}
We define $D_\phi(x, u)$ to be the value of \eqref{eq:inf_gen_phi1} evaluated at $x$ with a control input $u$. 
Here, $\gL_v$ is the Lie derivative along a vector or matrix field $v$.\footnote{The Lie derivative \((\gL_g\phi)\) along a matrix field \(g\) is interpreted as a row vector such that \((\gL_g\phi)u = (gu)\cdot\nabla\phi\).} The infinitesimal generator can be interpreted as the stochastic counterpart of the Lie derivative: It gives the derivative along the flow of $dX_t$ in expectation, \ie
\begin{align}
\mE[ d\phi(X_t)] = A\phi(X_t)dt .    
\end{align}

A desirable property for forward invariance is  $A\phi(X_t)$ being non-negative when $X_t \in \gC$ is close to the boundary $\partial \gC$ (Property 1). A desirable property for fast recovery is $A\phi(X_t)$ being positive when $X_t$ is outside of $\gC$ (Property 2). This property has motivated many safe control algorithms. We elaborate some of these algorithms to which our analysis can be applied.\footnote{Notice that the control policies discussed below are not mutually exclusive. That is, the control policies from Section \ref{sec:zero_CBF}, for some parameter choice, may be identical to those from Section \ref{sec:Gradient}.} 

\subsubsection{Zero control barrier function (zero-CBF)} 
\label{sec:zero_CBF}
Algorithms based on zero-CBF ensures that the control action satisfies
\begin{equation}\label{eq:zero_CBF_stochastic}
    A\phi(X_t) = D_\phi(X_t,U_t) \geq -\alpha(\phi(X_t))
\end{equation}
for all times \(t\) (see \cite{ames2019control,clark2019control,fan2019} and references therein). Here, $\alpha$ is chosen so that the desirable properties 1 and 2 holds: $A\phi(X_t)$ is non-negative at $X_t \in \partial \gC$ and strictly positive when $X_t \notin \gC$. For example, the use of a class $\mathcal K$ function\footnote{A continuous function \(\alpha: [0,a)\rightarrow[0,\infty]\) is defined as a class $\gK$ function if \(\alpha\) is strictly increasing and \(\alpha(0)=0\).} for $\alpha$ will yield these two properties. Although the use of class \(\gK\) function for \(\alpha\) is desirable, the analysis of this paper do not require this condition.

Condition~\eqref{eq:zero_CBF_stochastic} is often combined with a nominal controller $U^n_t = N(X_t)$ as follows: set \(U_t=U^n_t=N(X_t)\) if \(U^n_t\) satisfies \eqref{eq:zero_CBF_stochastic}, otherwise modify the control signal $U_t$ to satisfy~\eqref{eq:zero_CBF_stochastic}. Although in the latter case there may exist multiple \(U_t\) that satisfy~\eqref{eq:zero_CBF_stochastic}, we assume that these \(U_t\) are chosen based on a fixed policy \(R(X_t,N(X_t))\) that produce unique output \(U_t\) satisfying~\eqref{eq:zero_CBF_stochastic}. So, the overall function is representable by a function of the form 
\begin{align}
\label{eq:ControlSignal_Ha}
K(x)=\begin{cases}
 N(x), &   D_\phi(x, N(x)) \geq -\alpha(\phi(x)),   \\
 R(x,N(x)), & \text{Otherwise}.
\end{cases}
\end{align}
The problem of finding an alternative safe control action is often posed as a static optimization problem, and their analytic solution $R(x,N(x))$ may be derived using KKT conditions.\footnote{Such problems are not an optimal control problem, which requires consideration into future costs, and thus solving such problems are often easier. For example, if the point in the set $\{u \in \R^m :  D_\phi(X_t,u) \geq -\alpha( \phi(X_t))\}$ that are closet to the signal from the nominal controller $N(X_t)$ is chosen, such as in~\cite{ames_control_2019} and references therein, the closed form solution of \eqref{eq:ControlSignal_Ha} is given in~\cite{wei2019safe}.}

\subsubsection{Gradient-based approach} 
\label{sec:Gradient}
Algorithms based on gradient\footnote{Precisely speaking, it is based on gradient $\nabla\phi$ with respect to some inner product structures, \eg \(\gL_g \phi(x) = g\cdot \nabla\phi\).} use control signals of the following form
\begin{equation}\label{eq:gradient}
   K (x ) = N(x) +  c( x ) (\gL_g \phi(x))^\intercal,
\end{equation}
where $c: \R^n\to\R_+$ is a non-negative function~(see \cite{gracia2013reactive,khatib1986real} and references therein). 
The first term is some nominal control signal, and the second term pushes $d\phi$ towards the positive direction for the purpose of improved safety. This can be observed from 
\begin{align*}
     A\phi  = \gL_f \phi + \gL_g \phi N
      + c \gL_g \phi (\gL_g\phi)^\intercal +{1\over 2}\tr\left(\sigma\sigma^\intercal\Hess{\phi}\right) 
\end{align*}
where the term $c \gL_g \phi(\gL_g\phi)^\intercal$ takes non-negative values. The values of \(c(x)\) are used to manipulate the aggressiveness of safe control. That is, \(c(x)\) is set to \(A\phi(x)\geq 0\) as the system state \(x\) approaches the boundary $\partial \gC$ of the safe set from inside, and is strictly positive when \(x\) is outside of the safe set \(\gC\).

\subsubsection{Safety Verification}
\label{sec:SafetyVerification}
Instead of generating the control signals, safety verification is also an application area of our methods (see~\cite{prajna2007framework} and references therein). The safety verification methods consider verifying the safety in stochastic system~\eqref{eq:x_trajectory} with $U_t \equiv 0$. Our methods can also be used in a special case of safety verification using barrier certificate by modifying how the safe set is defined with respect to the barrier function $\phi$ and the definitions of~\eqref{eq:min_phi}, \eqref{eq:first_exit_time}, \eqref{eq:max_phi} and~\eqref{eq:first_entry_time}.\footnote{Barrier certificate considers the case when the initial state originates from a set~\cite{prajna2007framework}, while we consider that the original state is known. It is possible to extend the results when the initial state is generated from a distribution.}




\subsection{Objective and scope of this paper}

In this paper, we build upon analysis tools from stochastic processes to answer the following question:

\begin{displayquote}
In system~\eqref{eq:x_trajectory} with the settings described in~\ref{sec:zero_CBF}, ~\ref{sec:Gradient}, and~\ref{sec:SafetyVerification}, what are the \textit{exact distributions} of~\eqref{eq:min_phi}, \eqref{eq:first_exit_time}, \eqref{eq:max_phi}, and~\eqref{eq:first_entry_time}?
\end{displayquote}


The distribution of \eqref{eq:min_phi} and~\eqref{eq:first_exit_time} will allow us to study the property of invariance from a variety of perspectives, such as the average failure time, the tail distributions of safety or loss of safety, and the mean and tail of the safety margin.
The distribution of \eqref{eq:max_phi} and~\eqref{eq:first_entry_time} will allow us to study the property of recovery from a variety of perspectives, such as the average recovery time, tail distribution of recovery, and the mean and tail distribution for recovery vs. crashes. 
Although this paper focuses on the settings described in~\ref{sec:SafeControl}.1-3, our analysis methods can be generalized to controllers of other forms.

\section{Safe Probability}
\label{sec:safety_thm}
In this section, we study the properties of forward invariance and safety. The former is through characterizing the distribution of the safety margin~\eqref{eq:min_phi} and first exit time \eqref{eq:first_exit_time} in Section~\ref{sec:complete_information_distribution}. The latter is through characterizing the distribution of the distance to the safe set \eqref{eq:max_phi} and the recovery time \eqref{eq:first_entry_time} in Section~\ref{sec:ForwardConvergence}. The results of Section~\ref{sec:complete_information_distribution} and Section~\ref{sec:ForwardConvergence} are proved in Section~\ref{sec:Proof_theorem1_3} and Section~\ref{sec:Proof_theorem2_4}.



The random variables of our interests, \eqref{eq:min_phi}, \eqref{eq:first_exit_time}, \eqref{eq:max_phi} and~\eqref{eq:first_entry_time},  are functions of $\phi(X_t)$. Since the stochastic dynamics of $d\phi(X_t)$ are driven by $X_t$, we consider the augmented state space
\begin{align}\label{eq:augumented_z}
        Z_t := \begin{bmatrix}\phi(X_t) \\ X_t
        \end{bmatrix}\in\R^{n+1}.
\end{align}
From It\^o's lemma, dynamics of $\phi(X_t)$ satisfies
\begin{equation}\label{eq:ito_lemma2}
\begin{split}
    d\phi(X_t)  = A\phi(X_t)dt + \gL_{\sigma}\phi(X_t)dW_t.
\end{split}
\end{equation}
Combining \eqref{eq:x_trajectory} and~\eqref{eq:ito_lemma2}, the dynamics of $Z_t$ are given by the SDE 
\begin{align}\label{eq:MainTheorem_z}
    dZ_t &= \ave(Z_t)\,dt + \var(Z_t)\,dW_t,
\end{align}
where the drift and diffusion parameters are
\begin{align}\label{eq:mu_sigma_prime}
 & \ave(z) = \begin{bmatrix}
 D_{\phi}(x,K(x))\\
 (f(x) + g(x)K(x))
\end{bmatrix}, && 
    \var(z) = \begin{bmatrix}
 \gL_\sigma\phi(x)\\
 \sigma(x)
\end{bmatrix}.
\end{align}
The mapping $K$ is either \eqref{eq:ControlSignal_Ha} or \eqref{eq:gradient}, depending on the choice of control policies.
In the next section, we present out main theorems. Although we focus on the settings of~\ref{sec:SafeControl}.1-3, the analysis framework we present can be applied to any closed-loop control systems whose augmented state space $Z$ in~\eqref{eq:augumented_z} is the solution to the SDE~\eqref{eq:MainTheorem_z} with appropriate regularity conditions.\(^{\ref{ft:FN-Solution}}\)

\subsection{Probability of Forward Invariance}\label{sec:complete_information_distribution}

When the state initiates inside the safe set, the distribution of the safety margin, $\minf_x(T)$ in \eqref{eq:min_phi}, is given below.

\begin{theorem}\label{thm:InvariantProbability_MainTheorem1}
Consider system~\eqref{eq:x_trajectory} equipped with the control policy~\eqref{eq:ControlSignal_Ha} or~\eqref{eq:gradient} with the initial state \(X_0 = x\). 
Let $z = [\phi(x), x^\intercal]^\intercal\in\R^{n+1}$.
Then, the complementary cumulative distribution function (CCDF) of \(\minf_x(T)\),\footnote{We use CDFs and CCDFs with a slight abuse of notation. The presence or absence of equality in the inequality conditions are chosen so that the obtained probability can be used to compute \eqref{eq:safe_prob_invariant} and \eqref{eq:safe_prob_convergent}. Also, note that if a random variable $Y$ has a PDF, then $\mP(Y\leq \ell) = \mP(Y< \ell)$, so the presence or absence of equality in the probability density function does not affect the computed probabilities.\label{ft:inequality}} 
\addtocounter{equation}{1}
    \begin{align}\tag{\theequation.A}
        \FwdInv(z,T;\ell)=\mP(\Phi_x(T)\geq \ell),\quad\ell\in\R,
    \end{align}
    is the solution to the initial-boundary-value problem of the convection-diffusion equation on the super-level set \(\{z\in\R^n: z[1]\geq \ell\}\)
    \begin{align}\tag{\theequation.B}
    \label{eq:CauchyProblem}
        \begin{cases}
         {\partial \FwdInv\over\partial T} = {1\over 2}\grad\cdot(D\grad \FwdInv) + \gL_{\rho - {1\over 2}\grad\cdot D}\FwdInv,& z[1]\geq \ell, T>0,\\
         \FwdInv(z,T;\ell) = 0,& z[1]< \ell, T>0,\\
         \FwdInv(z,0;\ell) = \mathbb{1}_{\{z[1]\geq\ell\}}(z),&z\in\R^{n+1},
        \end{cases}
    \end{align}
    where \(D:=\zeta\zeta^\intercal\).\footnote{Here, the spatial derivatives \(\grad{}\) and \(\gL\) apply to the variable \(z\in\R^{n+1}\).\label{ft:spatial-derivative}}
\end{theorem}

Theorem~\ref{thm:InvariantProbability_MainTheorem1} converts the  problem of characterizing the distribution function into a deterministic convection-diffusion problem.
When we set \(\ell=0\), it gives the exact probability that $\gC$ is a forward invariant set with respect to~\eqref{eq:x_trajectory} during the time interval $[0,T]$. For arbitrary $\ell > 0$, it can be used to compute the probability of maintaining a safety margin of $\ell$ during $[0, T]$.

Moreover, the distribution of the first exit time from the safe set, $\exit_x(0)$ in \eqref{eq:first_exit_time}, and the first exit time from an arbitrary super level set \(\{x\in\R^n: \phi(x) \geq \ell\}\), is given below. 

\begin{theorem}\label{thm:InvariantProbability_MainTheorem2}
Consider system~\eqref{eq:x_trajectory} equipped with the control policy~\eqref{eq:ControlSignal_Ha} or~\eqref{eq:gradient} with the initial state \(X_0 = x\). Let $z = [\phi(x), x^\intercal]^\intercal\in\R^{n+1}$. Then, the cumulative distribution function (CDF) of the first exit time \(\exit_x(\ell)\)\footref{ft:inequality},
\addtocounter{equation}{1}
\begin{align}\tag{\theequation.A}
    \FwdInvExit(z,t; \ell) = \mP (\exit_x(\ell) \leq t ),
\end{align}
is the solution to
\begin{align*}\tag{\theequation.B}\label{eq:MainTheorem2}
\begin{cases}
    \frac{\partial \FwdInvExit}{\partial t} =  \frac{1}{2}\grad\cdot(D\grad{\FwdInvExit}) + \gL_{\rho-{1\over 2}\grad{ \cdot} D}\FwdInvExit
    ,& z[1]\geq \ell, t>0,\\
    \FwdInvExit(z,t) = 1,  & z[1] < \ell, t>0.\\
    \FwdInvExit(z,0) = \mathbb{1}_{\{z[1]<\ell\}}(z), & z\in\R^{n+1},\\
\end{cases}
\end{align*}
where $D = \zeta\zeta^\intercal$.\footref{ft:spatial-derivative}
\end{theorem}
Similarly, Theorem~\ref{thm:InvariantProbability_MainTheorem2} gives the distribution of the first exit time from an arbitrary super level set of barrier function as the solution to a deterministic convection-diffusion equation. When we set \(\ell=0\), it can also be used to compute the exact probability of staying within the safe set during any time interval.

\subsection{Probability of Forward Convergence}
\label{sec:ForwardConvergence}
When the state initiates outside the safe set, the distribution of the distance from the safe set, \(\maxf_x(T)\) in~\eqref{eq:max_phi}, is given below. 
\begin{theorem}\label{thm:ConvergenceProbability_MainTheorem3}
Consider system~\eqref{eq:x_trajectory} equipped with the control policy~\eqref{eq:ControlSignal_Ha} or~\eqref{eq:gradient} with the initial state \(X_0 = x\). Let $z = [\phi(x), x^\intercal]^\intercal\in\R^{n+1}$. Then, the CDF of \(\maxf_x(T)\),\footref{ft:inequality}   
\addtocounter{equation}{1}
\begin{equation}\tag{\theequation.A}
    \FwdConv(z,T;\ell) = \mP\left(\maxf_x(T) < \ell \right), \quad \ell \in \R,
\end{equation}
is the solution to 
    \begin{align}\tag{\theequation.B}
    \label{eq:pde_thm3}
        \begin{cases}
         {\partial \FwdConv\over\partial T} = {1\over 2}\grad\cdot(D\grad \FwdConv) + \gL_{\rho - {1\over 2}\grad\cdot D}\FwdConv,& z[1] < \ell, T>0,\\
         \FwdConv(z,T;\ell) = 0,& z[1]\geq\ell, T>0,\\
         \FwdConv(z,0;\ell) = \mathbb{1}_{\{z[1]<\ell\}}(z),&z\in\R^{n+1},\\
        \end{cases}
    \end{align}
where $D = \var\var^\intercal$.\footref{ft:spatial-derivative} 
\end{theorem}
Theorem~\ref{thm:ConvergenceProbability_MainTheorem3} gives the distribution of safety distance \(\maxf_x(T)\) as a deterministic convection-diffusion problem. When we set \(\ell=0\), it provides the exact probability that a stochastic process~\eqref{eq:x_trajectory} initiating outside the safe set enters the safe set during the time interval $[0,T]$. When we set $\ell<0$, it can be used to compute the probability of coming close to $|\ell|$-distance from the safe set along the time interval [0, T].

Finally, the distribution of the recovery time from the unsafe set, \(\entrance_x(0)\) in~\eqref{eq:first_entry_time}, and the entry time to an arbitrary super level set of the barrier function, \(\entrance_x(\ell)\), is given below. 
\begin{theorem}\label{thm:ConvergenceProbability_MainTheorem4}
Consider system~\eqref{eq:x_trajectory} equipped with the control policy~\eqref{eq:ControlSignal_Ha} or~\eqref{eq:gradient} with the initial state \(X_0 = x\). Let $z = [\phi(x), x^\intercal]^\intercal\in\R^{n+1}$. Then, the CDF of the recovery time \(\entrance_x(\ell)\),\footref{ft:inequality}  
\addtocounter{equation}{1}
\begin{align}\tag{\theequation.A}
    \FwdConvExit(z,t: \ell)=\mP ( \entrance_x(\ell) \leq t ),
\end{align}
is the solution to
\begin{align}\tag{\theequation.B}
\begin{cases}
    \frac{\partial \FwdConvExit}{\partial t}(z,t) =  \frac{1}{2}\grad\cdot(D\grad{\FwdConvExit}) + \gL_{\ave-{1\over 2} \grad{\cdot D}}{\FwdConvExit}
    ,& z[1] < \ell,t>0,\\
    \FwdConvExit(x,t) = 1,  & z[1]\geq \ell,t>0,\\
    \FwdConvExit(x,0) = \mathbb{1}_{\{z[1]\geq \ell\}}(z),& z\in\R^{n+1},\\
\end{cases}
\end{align}
where $D =\var\var^\intercal$.\footref{ft:spatial-derivative}
\end{theorem}

Theorem~\ref{thm:ConvergenceProbability_MainTheorem4} gives the distribution of the re-entry time as a solution to a deterministic convection-diffusion equation. When we set \(\ell=0\), it can be used to compute the exact probability of entering the safe region during any time intervals. When we set $\ell < 0$, it can be used to study the first time to reach \(|\ell|\)-close to the safe set.\footref{ft:distance_meansure}

\subsection{Proof of Theorem~\ref{thm:InvariantProbability_MainTheorem1} and Theorem~\ref{thm:ConvergenceProbability_MainTheorem3}}
\label{sec:Proof_theorem1_3}
The proofs of Theorem~\ref{thm:InvariantProbability_MainTheorem1} and Theorem~\ref{thm:ConvergenceProbability_MainTheorem3} are based on the lemmas presented below. These lemmas relate the distributions of various functionals of a diffusion process \(X_t\in\R^n, t\in\R_+\) with deterministic convection-diffusion equations. Let \(X_t\in\R^n, t\in\R_+\) be a solution of the SDE
\begin{equation}\label{eq:Wiener_drift_diffusion}
    d X_{t}=\mu(X_t) d t+\sigma(X_t) d W_{t},
\end{equation}
We assume sufficient regularity in the coefficients of~\eqref{eq:Wiener_drift_diffusion} such that \eqref{eq:Wiener_drift_diffusion} has a unique solution.\footref{ft:FN-Solution} 

We first present a variant of the Feynman--Kac Representation~\cite[Theorem 1.3.17]{pham2009continuous}.

\begin{lemma}[Feynman--Kac Representation]\label{thm:Invariant Probability}
Consider the diffusion process~\eqref{eq:Wiener_drift_diffusion} with $X_0=x \in \R^n$. Let $V:\R^n\times \R_+ \to \R$, $\beta: \R^n\times \R_+\to \R$ and \(\psi\colon\R^n\to\R\)
be some given functions.  The solution of the Cauchy problem 
\begin{align}
\begin{cases}
    \frac{\partial \lemmaF}{\partial T} 
    =  
    \textstyle\frac{1}{2}\tr(\sigma\sigma^\intercal\Hess{\lemmaF}) + \gL_\mu\lemmaF
    - V\lemmaF + \beta,
 &T>0,
    \\
    \lemmaF(x,0) = \psi(x) 
\end{cases}
\end{align}
is given by
\begin{align}
\begin{split}
\lemmaF(x,T) &= \mE_x\Biggl[
\textstyle
e^{-\int_0^T V(X_t,T-t)\, dt}
\psi(X_T)\\
&\left. + 
\int_0^T
\textstyle e^{-\int_0^\tau V(X_t,T-t)\, dt}
\beta(X_\tau,T-\tau)\, d\tau \right].
\end{split}
\end{align}
\end{lemma}

While the original statement given in~\cite[Theorem 1.3.17]{pham2009continuous} presents \(\lemmaF\) as a function of \(t\) that moves backward in time, we present \(\lemmaF\) as a function of \(T\) that moves forward in time.  Even though both approaches express the time derivative with opposite signs, both yield the same probability. The representation in Lemma~\ref{thm:Invariant Probability} can be used to derive the following result. 

\begin{lemma}\label{lem:Lemma1}
Let \(\gM\subset\R^n\) be a domain.  Consider the diffusion process \eqref{eq:Wiener_drift_diffusion} with the initial condition \(X_0=x\in\R^n\).
Then the function \(\lemmaF:\R^n\times [0,\infty)\to \R\) defined by
\begin{align}
U(x,T):=\mP_x(X_t\in\gM, \forall t\in [0,T])
\end{align}
is the solution to the convection-diffusion equation
\begin{align}
\label{eq:Lemma1CDE}
    \begin{cases}
     {\partial \lemmaF\over\partial T} = {1\over 2}\tr(\sigma\sigma^\intercal\Hess\lemmaF) + \gL_\mu \lemmaF,
     & x\in\gM, T>0,\\
     \lemmaF(x,T) = 0,&x\notin\gM, T>0,\\
     \lemmaF(x,0) = \mathbb{1}_{\gM}(x),& x\in\R^n.
    \end{cases}
\end{align}
\end{lemma}


\begin{proof}

Applying Lemma~\ref{thm:Invariant Probability} with
\begin{align*}
    \beta(x,t) \equiv 0,\, \psi(x) = \mathbb{1}_{\gM}(x),\,
    V(x,t) = \gamma\mathbb{1}_{\gM^c}(x),
\end{align*}
we have
\begin{align}
\label{eq:e2p_0}
    \lemmaF(x,T)
& = \mE_x\left[ e^{-\int^T_0 \gamma \mathbb{1}_{\gM^c}(X_t)\,dt}\psi(X_T)\right].
\end{align}
and
\begin{align}
\label{eq:eq_1}
    \begin{cases}
    \frac{\partial \lemmaF}{\partial T} = \frac{1}{2}\tr(\sigma\sigma^\intercal\Hess{\lemmaF}) + \gL_\mu\lemmaF,&x\in\gM, T>0,\\
    \frac{\partial \lemmaF}{\partial T} = \frac{1}{2}\tr(\sigma\sigma^\intercal\Hess{\lemmaF}) + \gL_\mu\lemmaF -\gamma\lemmaF,&x\notin \gM, T>0,\\
    \lemmaF(x,0) = \mathbb{1}_{\gM}(x),& x\in\R^n.
\end{cases}
\end{align}

Now, we take the limit of \(\gamma\to\infty\).  Since
\begin{align*}
    \lim_{\gamma \rightarrow \infty} e^{-\int^T_0 \gamma \mathbb{1}_{\gM^c}(X_t)\,dt } = 
\begin{cases}
0, & \text{if } X_t \notin \gM , \exists t\in[0,T],\\
1, & \text{otherwise}.
\end{cases},
\end{align*} \eqref{eq:e2p_0} becomes
\begin{align}\label{eq:Lemma1ProofFN-Solution}
    U(x,T)&= \mE_x[\mathbb{1}_{\{ X_t \in \gM , \forall t\in[0,T] \}}] \\
    & = \mP_x(X_t \in \gM , \forall t\in[0,T]).
\end{align}
Meanwhile, under the limit of \(\gamma\to\infty\), the part of \eqref{eq:eq_1} with \(x\notin \gM\) and \(T>0\)
reduces to the algebraic condition
\begin{align}\label{eq:Lemma1ProofCon}
    \lemmaF(x,T) = 0,\quad x\notin \gM, T>0.
\end{align}
Combining~\eqref{eq:eq_1}, \eqref{eq:Lemma1ProofFN-Solution}, and ~\eqref{eq:Lemma1ProofCon} and gives Lemma~\ref{lem:Lemma1}
\end{proof}

Using Lemma~\ref{lem:Lemma1}, we can prove Theorem~\ref{thm:InvariantProbability_MainTheorem1} as follows.
\begin{proof}[Proof (Theorem~\ref{thm:InvariantProbability_MainTheorem1})]
Consider the augmented space space $Z_t = [ \phi(X_t), X_t^\intercal]^\intercal \in \R^{n+1}$. The stochastic process \(Z_t\) is the solution to~\eqref{eq:MainTheorem_z} with parameters $\ave$ and $\var$ defined in~\eqref{eq:mu_sigma_prime} with the initial state
\begin{align}
    Z_0 = z = [\phi(x),x^\intercal]^\intercal.
\end{align}
Let us define 
\begin{align}
\label{eq:set_thm1} 
    \gM = \{ z \in \R^{n+1} : z[1] \geq \ell \} .
\end{align}
From Lemma~\ref{lem:Lemma1},
\begin{align*}
    F(x,T;\ell) = \mP_x( z_t \in \gM , \forall t \in [ 0, T] )
\end{align*}
is the solution to the convection-diffusion equation \begin{align}
\label{eq:lem1_pde}
\begin{cases}
{\partial \FwdInv \over\partial T} = {1\over 2}\operatorname{tr}(\zeta\zeta^\intercal\operatorname{Hess}\FwdInv) + \gL_\rho \FwdInv,& z\in \gM, T>0,\\
\FwdInv(z,T) = 0,& z\notin\gM,T>0,\\
\FwdInv(z,0) = \mathbb{1}_{\gM}(z),&z\in\R^{n+1}.
\end{cases}
\end{align}
To obtain  \eqref{eq:CauchyProblem}, define \(D = \zeta\zeta^\intercal\) and apply the vector identity\footnote{The identity is a direct consequence of the Leibniz rule \(\partial_i(D_{ij}\partial_j F) = (\partial_iD_{ij})\partial_j F + D_{ij}\partial_i\partial_j F.\)}
\begin{align}
\label{eq:VectorIdentity}
    \nabla\cdot (D\grad\FwdInv) = \gL_{\nabla\cdot D}\FwdInv + \tr(D\Hess\FwdInv).
\end{align}
\end{proof}
\begin{proof}[Proof (Theorem~\ref{thm:ConvergenceProbability_MainTheorem3})]
Consider the augmented space of $Z_t = [ \phi(X_t), X_t^\intercal]^\intercal \in \R^{n+1}$ in~\eqref{eq:augumented_z} in~\eqref{eq:augumented_z}. The stochastic process $Z_t$ is a solution of~\eqref{eq:MainTheorem_z} with parameters $\ave$ and $\var$ defined in~\eqref{eq:mu_sigma_prime} with the initial state  
\begin{align}
\label{eq:thm3-initial}
    Z_0 = z  = [ \phi(x), x^\intercal ]^\intercal.
\end{align}
Let us define 
\begin{align}
\label{eq:set_thm3} 
    \gM = \{ z \in \R^{n+1} : z[1] < \ell \} .
\end{align}
The CDF of $\maxf_x(T)$ is given by
\begin{align}
    \mP(\maxf_x(T) <\ell ) &= \mP_x(\forall t\in[0,T], \phi(X_t)<\ell)\\
    &=\mP_x(\forall t\in[0,T],Z_t\in\gM),
\end{align}
which, by Lemma~\ref{lem:Lemma1}, is the solution to the convection-diffusion equation \eqref{eq:Lemma1CDE}, yielding \eqref{eq:pde_thm3} after the application of identity \eqref{eq:VectorIdentity}.
\end{proof}

\subsection{Proof of Theorem~\ref{thm:InvariantProbability_MainTheorem2} and Theorem~\ref{thm:ConvergenceProbability_MainTheorem4}}
\label{sec:Proof_theorem2_4}
The proof of Theorem~\ref{thm:InvariantProbability_MainTheorem2} and Theorem~\ref{thm:ConvergenceProbability_MainTheorem4} requires Lemma~\ref{lem:EscapeTime} and Lemma~\ref{lem:EscapeTimeLemma3}. We will first present
these lemmas 
and then prove Theorem~\ref{thm:InvariantProbability_MainTheorem2} and Theorem~\ref{thm:ConvergenceProbability_MainTheorem4}.




\begin{lemma}\label{lem:EscapeTime}
Consider the diffusion process \eqref{eq:Wiener_drift_diffusion} with the initial point \(X_0=x\). Define the escape time as~\eqref{eq:EscapeTime}. 
Let \(\psi(x), V(x)\) be continuous functions and \(V\) be non-negative. If $\lemmaF: \R^{n} \to \R$ is the bounded solution to the boundary value problem
\begin{align}\label{eq:Col_boundary_condition}
    \begin{cases}
    {1\over 2}\tr(\sigma\sigma^\intercal \Hess{\lemmaF}) + \gL_\mu\lemmaF - V\lemmaF =0, & x\in\gM,\\
    \lemmaF(x) = \psi(x),& x\notin \gM,\\
    \end{cases}
\end{align}
then
\begin{align}
\label{eq:TimeIndependentSolution}
    \lemmaF(x) = \mE_x\left[\psi(X_{\lemmaT_{\gM}})e^{-\int_0^{\lemmaT_{\gM}}V(X_{s})\, ds}\right].
\end{align}
\end{lemma}


\begin{proof}
We first define a mapping $\eta: \R \rightarrow \R$ by 
 \begin{equation}\label{eq:67}
     \eta(s) := \lemmaF(X_s)e^{-\int_0^s V(X_v)\,dv}.
 \end{equation}
It satisfies
\begin{align}
\label{eq:noise_eta0}
    \eta(q)  &= \int_0^q d\eta(s) + \eta(0)\\
    \nonumber
    &=\int_0^q e^{-\int_0^s V(X_v)\, dv}\bigl[
    -V(X_s)\lemmaF(X_s) + \gL_\mu \lemmaF(X_s)\\
    \label{eq:noise_eta1}
    &\quad+\textstyle{1\over 2}\tr \left(\sigma(X_s)\sigma^\intercal(X_s)\Hess \lemmaF(X_s)\right)
    \bigr]ds \\
    \nonumber
    &\quad + \int_0^q e^{-\int_0^s V(X_v)\, dv}\gL_\sigma U(X_s)\, dW_s + \eta(0) \\
    &=\int_0^q
    e^{-\int_0^s V(X_v)\, dv}\gL_\sigma U(X_s)\, dW_s + \eta(0) ,
    \label{eq:noise_eta}
\end{align}
where \eqref{eq:noise_eta1} is from It\^o's Lemma; \eqref{eq:noise_eta} is from \eqref{eq:Col_boundary_condition} with $x \in \gM$. Thus, its expectation satisfies
\begin{align}
 \label{eq:eta_F1}
    \mE_x[\eta(q)] & = \mE_x[\eta(0)] \\
    \label{eq:eta_F2}
    &=  \mE_x[\lemmaF(X_0)] \\
    &= \lemmaF(x).
     \label{eq:eta_F}
\end{align}
where \eqref{eq:eta_F1} holds because the right hand side of \eqref{eq:noise_eta} has zero mean; \eqref{eq:eta_F2} is due to \eqref{eq:67}; and \eqref{eq:eta_F} is from the assumption that $X_0 = x$. 

Next, we set
\begin{align}
\label{eq:q-value}
   q = H_{\gM} .
\end{align}
As \eqref{eq:q-value} implies \(X_{q}\notin\gM.\), condition \eqref{eq:Col_boundary_condition} with $x \notin \gM$ yields  
\begin{align}
\label{eq:UGetsBdyValue1}
    \lemmaF(X_q) = \psi(X_q). 
\end{align}
Finally, we have 
\begin{align}
\label{eq:thm4_ucondition1}
    \lemmaF(x)&= \mE_x[\eta(q)]\\
    \label{eq:thm4_ucondition2}
    &=\mE_x\left[\lemmaF(X_q)e^{-\int_0^q V(X_v)\,dv} \right]\\
    \label{eq:thm4_ucondition3}
    &= \mE_x\left[\psi(X_q)e^{-\int_0^q V(X_v)\,dv }\right]\\
    \label{eq:thm4_ucondition4}
     &= \mE_x\left[\psi(X_{H_{\gM}})e^{-\int_0^{H_{\gM}} V(X_v)\,dv }\right],
\end{align}
where \eqref{eq:thm4_ucondition1} is due to \eqref{eq:eta_F}; \eqref{eq:thm4_ucondition2} is due to \eqref{eq:67}; \eqref{eq:thm4_ucondition3} is due to \eqref{eq:UGetsBdyValue1}, and \eqref{eq:thm4_ucondition4} is due to \eqref{eq:q-value}. 

\end{proof}

\begin{lemma}\label{lem:EscapeTimeLemma3}
Consider the diffusion process \eqref{eq:Wiener_drift_diffusion} with the initial condition \(X_0=x\in\R^n\).
Define the escape time,
\begin{align}\label{eq:EscapeTime}
    \lemmaT_{\gM}:=\inf\{t\in \R_+: X_t\notin\gM\}.
\end{align}
The CDF of the escape time
\begin{align}
    \lemmaF(x,t) = \mP_x(\lemmaT_\gM\leq t),\quad t>0,
\end{align}
is the solution to
\begin{align}\label{eq:EscapeTimeLemma3_boundary_condition}
    \begin{cases}
    {\partial \lemmaF \over \partial t}(x,t) = {1\over 2}\tr(\sigma\sigma^\intercal  \Hess{\lemmaF}) + \gL_\mu\lemmaF, & x\in\gM,t>0,\\
    \lemmaF(x,t) = 1, & x\notin\gM,t>0,\\
    \lemmaF(x,0) = \mathbb{1}_{\gM^c}(x), & x\in\R^n.\\
    \end{cases}
\end{align}
\end{lemma}

\begin{proof}
Let \(\gamma\in\Bbb C\) be a spectral parameter with \(\operatorname{Re}(\gamma)> 0\). For each fixed \(\gamma\), let $\hat{\lemmaF}(\cdot,\gamma): \R^n \to \Bbb C$ be the solution of
\begin{align}\label{eq:InvariantProbability_FeynmaKac_PDE_complex}
\begin{cases}
         \frac{1}{2}\tr(\sigma\sigma^\intercal\Hess{\hat{\lemmaF}}) + \gL_\mu\hat{\lemmaF} -\gamma \hat{V}\hat{\lemmaF}=0, & x\in\gM,\\
         \hat{\lemmaF}(x,\gamma) = \hat{\psi}(x,\gamma), & x\notin\gM .\\
\end{cases}
\end{align}
According to Lemma~\ref{lem:EscapeTime}, \(\hat{\lemmaF}(x,\gamma)\) is given by
\begin{align}\label{eq:78}
    \hat \lemmaF(x,\gamma) & = \mE_x\left[\hat{\psi}(X_{H_{\gM}})e^{-\gamma\int_0^{H_{\gM}}\hat{V}(X_s)\, ds}\right],
\end{align}
where \(X_s\) is the diffusion process in~\eqref{eq:Wiener_drift_diffusion}. Now, take
\begin{align}\label{eq:SpecialPsiV}
    \hat{\psi}(x,\gamma) = 1/\gamma,\quad \hat{V}(x) = 1,
\end{align}
in~\eqref{eq:InvariantProbability_FeynmaKac_PDE_complex} and~\eqref{eq:78}.
Then,~\eqref{eq:InvariantProbability_FeynmaKac_PDE_complex} becomes
\begin{align}
\label{eq:LaplaceDomainPDE}
\begin{cases}
         \frac{1}{2}\tr(\sigma\sigma^\intercal\Hess{\hat{\lemmaF}})+ \gL_\mu\hat{\lemmaF}-\gamma\hat{\lemmaF}=0, & x\in\gM,\\
         \hat{\lemmaF}(x,\gamma) = 1/ \gamma, & x\notin\gM,\\
\end{cases}
\end{align}
and \eqref{eq:78} becomes
\begin{align}
\nonumber
    \hat \lemmaF(x, \gamma) & = \mE_x\left[\tfrac{1}{\gamma}e^{-\gamma H_{\gM}}\right]\\
    \nonumber
    & = \int_0^\infty {1\over\gamma} e^{-\gamma t}p_{H_{\gM}\mid X_0}(t\mid x)\, dt\\
    \label{eq:LaplaceCDF3}
    &=\int_0^\infty e^{-\gamma t}\mP_x(H_{\gM}\leq t)\, dt
\end{align}
where integration by parts is used in \eqref{eq:LaplaceCDF3}.
Here, 
\(
    p_{H_{\gM}|X_0}(t|x) = {d \over dt}\mP_x(H_{\gM}\leq t)
\)
denotes the probability density function of \(H_\gM\) conditioned on \(X_0 = x\). 

Now, for fixed \(x\), let \(\lemmaF(x,t)\)  be the inverse Laplace transformation of \(\hat{\lemmaF}(x,\gamma)\). In other words,
\begin{align}\label{eq:Laplace}
    \hat \lemmaF(x,\gamma) & = \int_0^\infty \lemmaF(x,t)e^{-\gamma t}\, dt.
\end{align}
On one hand, comparing~\eqref{eq:LaplaceCDF3} and~\eqref{eq:Laplace}  gives
\begin{align}
    \lemmaF(x,t) = \mP_x\left(H_{\gM} \leq t\right).
\end{align}
On the other hand,
taking the inverse Laplace transformation of the PDE~\eqref{eq:LaplaceDomainPDE} for \(\hat \lemmaF\) yields
the PDE satisfied by \(\lemmaF\):
\begin{align}
\begin{cases}
         {\partial \lemmaF \over \partial t}(x,t)=\frac{1}{2}\tr(\sigma\sigma^\intercal\Hess{\lemmaF}) + \gL_\mu\lemmaF, & x\in\gM,t>0,\\
         \lemmaF(x,t) = 1, & x\notin\gM,t>0,\\
        \lemmaF(x,0) = \mathbb{1}_{\gM^c}(x), &  x\in\R^n.
\end{cases}
\end{align}
\end{proof}


Now, we are ready to prove Theorem~\ref{thm:InvariantProbability_MainTheorem2} and Theorem~\ref{thm:ConvergenceProbability_MainTheorem4}.
\begin{proof}[Proof (Theorem~\ref{thm:InvariantProbability_MainTheorem2})]
Let the augmented state \(Z_t\) be the solution of~\eqref{eq:MainTheorem_z} with the initial state 
    \(Z_0 = z = [\phi(x), x^\intercal]^\intercal\in\R^{n+1}.\)
Let 
\begin{align}
    \gM = \{z\in\R^{n+1}: z[1]\geq \ell\}.
\end{align}
Thus, we have \(\exit_x(\ell) = H_{\gM}\), for \(\FwdInvExit\) be the solution of the convection-diffusion equation with \(D = \zeta\zeta^\intercal\)
\begin{align*}\
\begin{cases}
    \frac{\partial \FwdInvExit}{\partial t} =  \frac{1}{2}\grad\cdot(D\grad{\FwdInvExit}) + \gL_{\rho-{1\over 2}\grad{ \cdot} D}\FwdInvExit
    ,& z\in\gM, t>0,\\
    \FwdInvExit(z,t) = 1,  &  z\notin\gM, t>0.\\
    \FwdInvExit(z,0) = \mathbb{1}_{\gM^c}(z),&z\in\R^{n+1}.
\end{cases}
\end{align*}
From Lemma~\ref{lem:EscapeTimeLemma3}, we have 
\(
   \FwdInvExit(x,t)  = \mP(\exit_{x}\leq t).
\)
\end{proof}

\begin{proof}[Proof (Theorem~\ref{thm:ConvergenceProbability_MainTheorem4})]
Let the augmented state space \(Z_t\) be the solution of~\eqref{eq:MainTheorem_z} with the initial state 
    \(Z_0 = z = [\phi(x), x^\intercal]^\intercal\in\R^{n+1}.\)
Let 
\begin{equation}
    \gM = \{z\in\R^{n+1}: z[1] < \ell\}.
\end{equation}
We have \(\entrance_x(\ell) = H_{\gM}\) for \(\FwdConvExit\) be the solution of the convection diffusion equation
\begin{align*}
\begin{cases}
    \frac{\partial \FwdConvExit}{\partial t}(z,t) =  \frac{1}{2}\grad\cdot(D\grad{\FwdConvExit}) + \gL_{\ave-{1\over 2} \grad{\cdot D}}{\FwdConvExit}
    ,& z\in\gM,t>0,\\
    \FwdConvExit(x,t) = 1,  & z\notin\gM,t>0,\\
    \FwdConvExit(x,0) = \mathbb{1}_{\gM^c}(z),&z\in\R^{n+1},
\end{cases}
\end{align*}
where \(D = \zeta\zeta^\intercal\).
From Lemma~\ref{lem:EscapeTimeLemma3}, we have 
\begin{align}
   \FwdConvExit(x,t)  = \mP(\entrance_{x}\leq t).
\end{align}
\end{proof}
\section{Conclusion}
This paper gives the exact probability distributions of the minimum and maximum barrier function values during any time interval and the first entry and exit times to and from any super level sets of the barrier function. The distributions of these variables can be used to study many safety criteria, including but not limited to the probability of safety and recovery, the safety margin, and the mean and tail distributions of the failure and recovery times. These results lay out the foundation for formulating new optimization problems with probabilistic bounds on safety and recovery and for solving these problems using existing techniques from PDE-constrained optimization. Ultimately, these tools can guide the design of a variety of safety-critical autonomous systems, including autonomous vehicles and robots~\cite{koopman2017autonomous,tadele2014safety,moustris2011evolution}. 

\bibliography{citation}

\end{document}